\documentclass[12pt]{article}
\usepackage{latexsym,amssymb,amsmath,amsthm,enumerate,geometry,cite,float,verbatim}
\newtheorem{theorem}{Theorem}
\newtheorem{lemma}[theorem]{Lemma}

\usepackage{tikz}
\usepackage{lineno}
\usepackage{setspace}
\allowdisplaybreaks

\begin{document}
\onehalfspace

\title{Efficiently recognizing graphs with equal\\ independence and annihilation numbers}
\author{Johannes Rauch\and Dieter Rautenbach}
\date{}

\maketitle
\vspace{-10mm}
\begin{center}
{\small 
Institute of Optimization and Operations Research, Ulm University, Ulm, Germany\\
\texttt{$\{$johannes.rauch,dieter.rautenbach$\}$@uni-ulm.de}
}
\end{center}

\begin{abstract}
The annihilation number $a(G)$ of a graph $G$ 
is an efficiently computable upper bound 
on the independence number $\alpha(G)$ of $G$.
Recently, Hiller observed that a characterization of the graphs $G$
with $\alpha(G)=a(G)$ due to Larson and Pepper is false.
Since the known efficient algorithm recognizing these graphs 
was based on this characterization, 
the complexity of recognizing graphs $G$ with $\alpha(G)=a(G)$
was once again open.
We show that these graphs can indeed be recognized efficiently.
More generally, we show that recognizing graphs $G$ with $\alpha(G)\geq a(G)-\ell$
is fixed parameter tractable using $\ell$ as parameter.\\[3mm]
{\bf Keywords:} Independence number; annihilation number
\end{abstract}

\section{Introduction}

We consider finite, simple, and undirected graphs, and use standard terminology.
Let $G$ be a graph with $n$ vertices, $m$ edges, 
and degree sequence $d_1\leq \ldots \leq d_n$.
If $d_1+\cdots+d_{a+1}>m$ for some integer $a$ in $\{ 1,\ldots,n-1\}$, 
then $G$ has no independent set of order $a+1$.
In fact, if $X$ is any set of $a+1$ vertices of $G$, 
then the degree sum $d_G(X)=\sum\limits_{u\in X}d_G(u)$ of $X$
is at least $d_1+\cdots+d_{a+1}$, and, hence,
strictly larger than the number $m$ of edges of $G$,
which implies that at least one of the edges of $G$ has both its endpoints in $X$.
This observation implies that the {\it annihilation number} $a(G)$ of $G$
defined by Pepper \cite{pe} as
$$a(G)=\max\big\{ a\in \{ 1,\ldots,n\}:d_1+\cdots+d_a\leq m\big\}$$
is an efficiently computable upper bound 
on the computationally hard \cite{ka} 
{\it independence number} $\alpha(G)$ of $G$,
that is,
$$\alpha(G)\leq a(G)\mbox{ for every graph $G$.}$$
Larson and Pepper \cite{lape} proposed a characterization 
of the graphs $G$ that satisfy $\alpha(G)=a(G)$
leading to a polynomial time recognition algorithm for these graphs.
Recently, Hiller \cite{hi} pointed out an error in the proof of this characterization,
and provided numerous counterexamples to its statement.
She showed however that the characterization of Larson and Pepper 
holds for bipartite graphs and connected claw-free graphs.
Since, for these two classes of graphs, 
the independence number can be determined efficiently \cite{sb},
bipartite or claw-free graphs $G$ with $\alpha(G)=a(G)$
can easily be recognized efficiently 
by simply computing and comparing both parameters.
By Hiller's contributions though, 
the complexity of recognizing general graphs $G$ with $\alpha(G)=a(G)$
is once again an open problem. 

We solve this problem by showing the following theorem.

\begin{theorem}\label{theorem1}
For a given graph $G$, 
one can determine in polynomial time whether $\alpha(G)=a(G)$.
\end{theorem}
Theorem \ref{theorem1} is actually a consequence of the following 
more general result.
\begin{theorem}\label{theorem2}
For a given pair $(G,\ell)$, where $G$ is a graph, and $\ell$ is a non-negative integer, 
one can decide whether $\alpha(G)\geq a(G)-\ell$
in time $O\left(15^{(3\ell+1)} \ell (n+m)^3\right)$.
\end{theorem}
The following section contains the proofs.
For Theorem \ref{theorem1}, we give a simple self-contained argument.
For Theorem \ref{theorem2}, 
we rely on deep results from the area of fixed parameter tractability.

\section{Proofs}

Let $G$ be a graph with $n$ vertices, $m>0$ edges, 
degree sequence $d_1\leq \ldots \leq d_n$, and
annihilation number $a$.
Since $d_1+\cdots+d_n=2m$ and
$d_1+\cdots+d_{\left\lceil \frac{n-1}{2}\right\rceil}\leq 
d_{\left\lceil \frac{n-1}{2}\right\rceil+1}+\cdots+d_n$,
we have $d_1+\cdots+d_{\left\lceil \frac{n-1}{2}\right\rceil}\leq m$,
and, hence, $a\geq \frac{n-1}{2}$.
This implies $k\geq 0$ for $k=a-\frac{n-1}{2}$.
Note that $2k$ is an integer.
Furthermore, $m>0$ implies $a\leq n-1$, and, hence, $k\leq \frac{n-1}{2}$.

By the definition of the annihilation number, we have $d_1+\cdots+d_{a+1}>m$,
which implies
\begin{eqnarray}\label{e1}
d_{a+2}+\cdots+d_n<m.
\end{eqnarray}
Let $I$ be a maximum independent set in $G$.
The set $R=V(G)\setminus I$ satisfies $|R|=n-|I|\geq n-a=\frac{n+1}{2}-k$.

\begin{lemma}\label{lemma1}
In the above setting, $|A|-|N_G(A)|\leq 2k$ for every subset $A$ of $I$.
\end{lemma}
\begin{proof}
Suppose, for a contradiction, that $|A|-|N_G(A)|\geq 2k+1$ for some subset $A$ of $I$.
Since $I$ is independent, we have that $N_G(A)\subseteq R$.
Since $A$ is independent,
the number of edges of $G$ between $A$ and $N_G(A)$
is equal to $d_G(A)$ as well as less or equal to $d_G(N_G(A))$,
which implies $d_G(A)\leq d_G(N_G(A))$.
Clearly, $|N_G(A)|\leq |A|-2k-1\leq |I|-2k-1\leq a-2k-1=\frac{n-1}{2}-k-1$.
Let $B$ be a subset of $R$ that contains $N_G(A)$, 
and has order exactly $n-a-1=\frac{n-1}{2}-k<|R|$.
By (\ref{e1}), the largest $|B|=n-a-1$ degrees sum to less than $m$,
which implies $d_G(B)<m$.
The set $C=(B\setminus N_G(A))\cup A$ satisfies
$d_G(C)=d_G(B)-d_G(N_G(A))+d_G(A)\leq d_G(B)<m$
and 
$|C|=|B|-|N_G(A)|+|A|\geq |B|+2k+1=\frac{n-1}{2}+k+1=a+1$,
that is, the degree sum of the at least $a+1$ vertices in $C$ is at most $m$,
which contradicts the definition of the annihilation number.
\end{proof}
We are now in a position for the following.

\begin{proof}[Proof of Theorem \ref{theorem1}]
Let $G$ be as specified at the beginning of this section.
Let $G'$ arise from $G$ by adding a set $S$ of $2k$ new vertices
as well as all possible new edges between $I$ and $S$.
By Lemma \ref{lemma1} and Hall's theorem \cite{ha},
the graph $G'$ has a matching $M'$ saturating all vertices in $I$,
and only containing edges between $I$ and $R\cup S$.

Now, suppose that $\alpha(G)=a$, that is, in particular, $|I|=a$.
Since $|S|=2k\leq \frac{n-1}{2}+k=|I|$,
we may assume, by construction, that $M'$ saturates all vertices in $S$.
Removing from $M'$ the $2k$ edges incident with the vertices in $S$
implies that $G$ has a matching $M\subseteq M'$ 
saturating all but exactly $2k$ vertices in $I$,
and only containing edges between $I$ and $R$.
Since $|R|=\frac{n+1}{2}-k=\frac{n-1}{2}+k-2k+1=|I|-2k+1=|M|+1$,
the matching $M$ saturates all but exactly one vertex, say $r$, in $R$.
Since $R\setminus \{ r\}$ is a vertex cover of $G-r$,
the matching $M$ is a maximum matching of $G-r$,
and every maximum matching of $G-r$
\begin{itemize}
\item saturates all but exactly $2k$ vertices in $I$,
\item saturates all vertices in $R\setminus \{ r\}$,
and 
\item only contains edges between $I$ and $R\setminus \{ r\}$.
\end{itemize}
Now, let $N$ be any maximum matching of $G-r$,
that is, $|N|=|M|=a-2k=\frac{n-1}{2}-k$.
Let 
$N=\left\{ u_1v_1,\ldots,u_{\frac{n-1}{2}-k}v_{\frac{n-1}{2}-k}\right\}$.
By the above observations, 
the set $I_0$ of vertices of $G-r$ that are not saturated by $N$ is a subset of $I$.
The choice of $N$ implies that $I_0$ is independent.
Now, we describe a $2$-{\sc Sat} formula $f$ 
such that every satisfying truth assignment for $f$
allows to reconstruct a suitable choice for $I$.
The formula $f$ uses the boolean variables $x_1,\ldots,x_{\frac{n-1}{2}-k}$,
where $x_i$ corresponds to the edge $u_iv_i$ in $N$,
$x_i$ being true corresponds to $u_i\in I$, and 
$x_i$ being false corresponds to $v_i\in I$.
The formula $f$ contains the following polynomially many clauses:
\begin{itemize}
\item For every $i\in \left\{ 1,\ldots,\frac{n-1}{2}-k\right\}$,
\begin{itemize}
\item if $v_i$ has a neighbor in $I_0$,
then we add to $f$ the clause $x_i$, and
\item if $u_i$ has a neighbor in $I_0$,
then we add to $f$ the clause $\bar{x}_i$.
\end{itemize}
\item For every two distinct indices $i,j\in \left\{ 1,\ldots,\frac{n-1}{2}-k\right\}$,
\begin{itemize}
\item if $u_i$ is adjacent to $u_j$,
then we add to $f$ the clause $\bar{x}_i\vee \bar{x}_j$,
\item if $u_i$ is adjacent to $v_j$,
then we add to $f$ the clause $\bar{x}_i\vee x_j$,
\item if $v_i$ is adjacent to $u_j$,
then we add to $f$ the clause $x_i\vee \bar{x}_j$, and
\item if $v_i$ is adjacent to $v_j$,
then we add to $f$ the clause $x_i\vee x_j$.
\end{itemize}
\end{itemize}
Clearly, setting $x_i$ to true if $u_i\in I$ and to false otherwise
for every $i\in \left\{ 1,\ldots,\frac{n-1}{2}-k\right\}$
yields a satisfying truth assignment for $f$.
Conversely, given a satisfying truth assignment for $f$, then 
$$
I'=I_0\cup 
\left\{ u_i:\mbox{$x_i$ is true and $i\in \left\{ 1,\ldots,\frac{n-1}{2}-k\right\}$}\right\}
\cup
\left\{ v_i:\mbox{$x_i$ is false and $i\in \left\{ 1,\ldots,\frac{n-1}{2}-k\right\}$}\right\}
$$
is an independent set of order $a$, and, hence,
a suitable choice for $I$.

Now, let $G$ be a given graph of order $n$ and size $m$,
for which we want to decide whether $\alpha(G)$ equals its annihilation number $a$.
If $m=0$, then $\alpha(G)=a=n$. Hence, we may assume that $m>0$.
First, we compute the annihilation number $a$ of $G$.
Now, we consider all $n$ choices for the vertex $r$.
For each choice of $r$, we determine a maximum matching $N$ of $G-r$ \cite{ed},
set up the corresponding $2$-{\sc Sat} formula $f$,
check the satisfiability of $f$, and 
determine a satisfying truth assignment provided that $f$ is satisfiable \cite{asplta}.
All these steps can be executed in polynomial time using standard algorithms.

If $|N|=\frac{n-1}{2}-k$ and $f$ is satisfiable,
then the set $I'$ as above is a maximum independent set certifying $\alpha(G)=a$.
Hence, if $\alpha(G)=a$, then this approach yields an independent set of order $a$
certifying $\alpha(G)=a$, and, if $\alpha(G)<a$, then, for every choice of $r$,
either $|N|<\frac{n-1}{2}-k$ 
or $|N|=\frac{n-1}{2}-k$ but $f$ is not satisfiable.
This completes the proof of Theorem \ref{theorem1}.
\end{proof}
We proceed to the following.

\begin{proof}[Proof of Theorem \ref{theorem2}]
Let $G$ be as specified at the beginning of this section.
Arguing as in the proof of Theorem \ref{theorem1},
we obtain that $G$ has a matching $M$ 
saturating all but at most $2k$ vertices in $I$,
and only containing edges between $I$ and $R$.
Suppose that $\alpha(G)\geq a-\ell$, that is, in particular, $|I|\geq a-\ell$.
Since 
$|R|-|M|
=n-|I|-|M|
\leq n-|I|-(|I|-2k)
=n-2|I|+2k
\leq n-2a+2\ell+2k
=2\ell+1$,
the matching $M$ covers all but at most $2\ell+1$ vertices in $R$.
In particular,
the matching number $\mu(G)$ of $G$ satisfies
$\mu(G)\geq |M|\geq |R|-(2\ell+1)=n-\alpha(G)-2\ell-1\geq n-a-2\ell-1$.
Now, 
$\alpha(G)\geq a-\ell$
if and only if 
the vertex cover number $\tau(G)=n-\alpha(G)$ of $G$ satisfies
$\tau(G)\leq n-a+\ell$.
Therefore, 
setting $p:=n-a+\ell-\mu(G)\leq n-a+\ell-n+a+2\ell+1=3\ell+1$,
we obtain that $\alpha(G)\geq a-\ell$
if and only if 
$\tau(G)\leq \mu(G)+p$.
Combining results from \cite{ed},\cite{mi}, and \cite{ra},
in time $O\left(15^{(3\ell+1)} \ell (n+m)^3\right)$,
one can determine $a$, $\mu(G)$, and $p$, 
check whether $p\leq 3\ell+1$
(if this is not the case, then $\alpha(G)<a-\ell$), 
and (provided that $p\leq 3\ell+1$)
decide whether $\tau(G)\leq \mu(G)+p$,
which completes the proof.
\end{proof}

\end{document}